\documentclass[11pt,a4paper,reqno]{amsart}

\usepackage[mathscr]{eucal}
\usepackage{amsmath,amsfonts,calrsfs,amssymb,color}
\usepackage{amsthm}

\usepackage[normalem]{ulem}

\newcommand\dela[1]{}

\newtheorem{Theorem}{Theorem}[section]

\newtheorem{Lemma}{Lemma}[section]
\newtheorem{Corollary}{Corollary}
\newtheorem{Remark}{Remark}[section]

\newcommand{\R}{\mathbb R}

\renewcommand{\P}{\mathbb P}

\newcommand{\ds}{\displaystyle}

\newcommand{\CO}{{{ \mathcal O }}}
\newcommand\divv{\mbox{div}}
\newcommand{\RR}{{\mathbb{R}}}
\newcommand{\DEQSZ}{\begin{eqnarray}}
\newcommand{\EEQSZ}{\end{eqnarray}}
\newcommand{\lk}{\left}
\newcommand{\rk}{\right}
\newcommand{\lb}{\langle}
\newcommand{\rb}{\rangle}
\newenvironment{acknowledgements}{\noindent{\bf Acknowledgements. }} {}

\newcommand{\PP}{{\mathbb{P}}} 
\makeatletter
\@addtoreset{equation}{section}

\makeatother

\begin{document}

\title[]{Existence and convergence results for infinite dimensional nonlinear stochastic equations with multiplicative noise}

\author{Viorel Barbu}
\address{\small University Al. I. Cuza and Institute of Mathematics Octav
Mayer,\\  \small Ia\c si, Romania}
\author{Zdzis{\l}aw Brze\'{z}niak}
\address{\small Department of Mathematics, University of York, \small
Heslington, York YO10 5DD, UK}
\author{Erika Hausenblas}
\address{\small
Department of Mathematics and Informationtechnology, Montanuniversity Leoben,\\
\small Franz Josefstr. 18, 8700 Leoben, Austria}
\author{Luciano Tubaro}
\address{\small Department of Mathematics, University of Trento, Italy }
\date{June \textcolor{blue}{19th},  2012}

\begin{abstract}
The solution $X_n$ to a nonlinear stochastic differential equation
of the form
$dX_n(t)+A_n(t)X_n(t)\,dt-\tfrac12\sum_{j=1}^N(B_j^n(t))^2X_n(t)\,dt=\sum_{j=1}^N B_j^n(t)X_n(t)d\beta_j^n(t)+f_n(t)\,dt$,
$X_n(0)=x$, where $\beta_j^n$ is a regular approximation of a
Brownian motion $\beta_j$, $B_j^n(t)$ is a family of linear continuous
operators from $V$ to $H$ strongly convergent to $B_j(t)$,
$A_n(t)\to A(t)$, $\{A_n(t)\}$ is a family of maximal monotone
nonlinear operators of subgradient type from $V$ to $V'$, is
convergent to the solution to the stochastic differential equation
$dX(t)+A(t)X(t)\,dt-\frac12\sum_{j=1}^NB_j^2(t)X(t)\,dt=\sum_{j=1}^NB_j(t)X(t)\,d\beta_j(t)+f(t) \,dt$,
$X(0)=x$. Here $V\subset H\cong H'\subset V'$ where $V$ is a reflexive
Banach space with dual $V'$ and $H$ is a Hilbert space. These
results can be reformulated in terms of Stratonovich stochastic
equation $dY(t)+A(t)Y(t)\,dt=\sum_{j=1}^NB_j(t)Y(t)\circ d\beta_j(t)+f(t)\,dt$.

\end{abstract}
\maketitle






\section{Introduction}

Consider the stochastic differential equation
\begin{equation}\label{e1.1}\lk\{
\begin{array}{l}
\begin{aligned}dX(t)+A(t)X(t)\,dt -\tfrac12 \sum_{j=1}^N B_j^2(t)X(t)\,dt=\hspace{4cm}\\
\sum_{j=1}^N B_j(t)X(t)\,d\beta_j(t)+f(t)\,dt,\quad t\in[0,T]\end{aligned}\\
X(0)=x
\end{array}\rk.
\end{equation}
where $A(t)\colon V\to V'$ is a nonlinear monotone operator,
$B_j(t)\in L(V,H)$, $\forall t\in[0,T]$ and $\beta_j$ are independent Brownian
motion in a probability space $\{\Omega,\mathcal F,\{\mathcal
F_t\}_{t\ge 0},\P\}$.

Equation \eqref{e1.1} is of course equivalent with the Stratonovich stochastic differential equation

$$
\leqno{(1.1)'} \qquad dY(t)+A(t)Y(t)\,dt=\sum_{j=1}^N B_j(t)Y(t)\circ d\beta_j(t)+f(t)\,dt,\quad t\in[0,T].
$$
Here $V$ is a reflexive Banach space with dual $V'$ such that $V\subset H\subset V'$ algebraically and topologically, where
$H$ (the {\it pivot} space) is a real Hilbert space. (The assumptions on $A(t)$, $B_j(t)$ will made precise later on.)

We associate with \eqref{e1.1} the random differential equation
\begin{equation}\label{e1.2}
\lk\{\begin{array}{l}
\ds\frac{dy}{dt}+\Lambda(t)y=g(t),\qquad t\in[0,T],\;\P\mbox{-a.s.}\\ \\
y(0)=x,
\end{array}\rk.
\end{equation}
where $g(t)=e^{\sum_{j=1}^N\beta_j(t)B_j^\ast(t)}f(t)$, $B_j^\ast(t)$ is the adjoint of $B_j(t)$ and
$\Lambda(t)$ is the family of operators
\begin{multline}\label{e1.3}
\Lambda(t)y=e^{-\sum_{j=1}^N\beta_j(t)\,B_j(t)}A(t)e^{\sum_{j=1}^N\beta_j(t)\,B_j(t)}y \;+\\
+\sum_{j=1}^N\int_0^{\beta_j(t)}e^{-s\,B_j(t)}\dot B_j(t)e^{s\,B_j(t)}y\,ds,
\qquad \forall y\in V
\end{multline}
where $\dot B_j$ is the derivative of $t\to B_j(t)\in L(V,H)$ and
$(e^{sB_j(t)})_{s\in\RR}$ is the $C_0$-group generated by $B_j(t)$ on
$H$ and $V$.

It is well known (see e.g., \cite[pag. 202]{DZ}, \cite{brz-cap-flan}, \cite{flan-lisei}
) that assuming that $B_jB_k(t)=B_k(t)B_j(t)$ for all $j$, $k$, at least formally equations \eqref{e1.1}
and \eqref{e1.2} are equivalent via the transformation
\begin{equation}\label{e1.4}
X(t)=e^{\sum_{j=1}^N\beta_j(t)\,B_j(t)}y(t),\qquad \P\mbox{-a.s}, \, t\in [0,T],
\end{equation}
and this is indeed the case if \eqref{e1.2} has a strong, progressively measurable solution
$y\colon [0,T]\times\Omega\to H$.

We consider also the family of approximating stochastic equations
\begin{equation}\label{e1.5}\lk\{
\begin{array}{l}
\ds\frac{d\ }{dt}X_n+A_n(t)X_n= \sum_{j=1}^NB_j^n(t)X_n(t)\dot\beta_j^n(t)+f_n(t),\qquad \P\mbox{-a.s.}\\ \\
X_n(0)=x.
\end{array}\rk.
\end{equation}
where $\{\beta_j^n\}$ is a sequence of {smooth} stochastic processes
convergent to $\beta_j$, that is $\beta_j^n(t)\to \beta_j(t)$ uniformly on $[0,T]$, $\P$-a.s. and $A_n\to A$,
 $B_j^n\to B_j$, $f_n\to f$ as $n\to \infty$ in a sense to be made precise below.

Equation \eqref{e1.5} is just an approximation of Stratonovich equation \eqref{e1.1}$'$
where $\beta_j^n$ is a regularization of $\beta_j$. One must emphasize that $\{\beta_n\}$
might be adapted to a filtration different of $\{\mathcal F_t\}$.\\
Equation \eqref{e1.5} reduces via \eqref{e1.4}, that is
$$
\leqno{(1.4)_n} \qquad X_n(t)=e^{\sum_{j=1}^N\beta_j^n(t)\,B_j^n(t)}y_n(t)
$$
to a random differential equation of the form \eqref{e1.2} that is
$$\leqno{(1.2)_n}\lk\{
\begin{array}{l}
\ds\frac{dy_n}{dt}+\Lambda_n(t)y_n=g_n(t),\qquad t\in[0,T],\;\P\mbox{-a.s.}\\ \\
y_n(0)=x
\end{array}\rk.
$$
where $g_n(t)=e^{\sum_{j=1}^N\beta_j^n(t){B_j^n}^\ast(t)}f_n(t)$ and $\Lambda_n$ are given by
$$\leqno{(1.3)_n}\quad
\begin{aligned}\Lambda_n(t)=e^{-\sum_{j=1}^N\beta_j^n(t)\,B_j^n(t)}A_n(t)e^{\sum_{j=1}^N\beta_j^n(t)\,B_j^n(t)}+\hspace*{2cm}\\
+\sum_{j=1}^N\ds\int_0^{\beta_j^n(t)}e^{-s\,B_j^n(t)}\dot B_j^n(t)e^{s\,B_j^n(t)}y\,ds.\end{aligned}
$$
The main result (see Theorems \ref{t2.1}, \ref{t2.2}, \ref{t2.3}
below) is that under suitable assumptions equations \eqref{e1.2},
(1.2)${}_n$ have unique solutions $y$ and $y_n$ which are
progressively measurable processes and for $n\to\infty$, we have
$y_n\to y$, $X_n\to X$  in a certain precise sense. In the linear case such a result was established
by a different method  in \cite{brz-cap-flan}, (we refer also to \cite{DaPT}, \cite{flan-lisei}
to other results in this direction.) The variational method we use here allows to treat a general class of nonlinear equations \eqref{e1.1}
possibly multi-valued (On these lines see also \cite{B,B2}.)

The applications given in Sect.\ 4.1 refer to stochastic porous
media equations and nonlinear stochastic diffusion equations of
divergence type but of course the potential area of applications
is much larger.
\paragraph{Notation}  If $Y$ is a Banach space we denote by $L^p(0,T;Y)$, $1\le p\le \infty$ the
space of all (equivalence classes of) $Y$-measurable functions $u\colon (0,T)\to Y$ with $\|u\|_Y\in L^p(0,T)$
(here $\|\cdot\|_Y$ is the norm  of $Y$). Denote by $C([0,T];Y)$ the space of all continuous $Y$-valued
functions on $[0,T]$ and by $W^{1,p}([0,T];Y)$ the infinite dimensional Sobolev space $\{y\in L^p(0,T;Y), \frac{dy}{dt}\in L^p(0,T;Y)\}$
where $\frac{d\ }{dt}$ is considered in sense of vectorial distributions. It is well known that $W^{1,p}([0,T];Y)$ coincides
with the space of absolutely continuous functions $y\colon [0,T]\to Y$, a.e. differentiable and with derivative
$y'(t)=\frac{dy}{dt}(t)$ a.e. $t\in (0,T)$ and $\frac{dy}{dt}\in L^p(0,T;Y)$(see e.g., \cite{barbu-2010}.)
If $p\in[1,\infty]$ is given we denote by $p'$ the conjugate
exponent, i.e., $\frac1p+\frac1{p'}=1$.

\def\temp1{Given a lower-semicontinuous and convex function $\varphi\colon Y\to \bar \R=]-\infty,+\infty]$ denote by
$\partial \varphi\colon Y\to Y'$ (the dual space) the \textbf{subdifferential} of $\varphi$, i.e.,
\begin{equation}\label{e1.6}
\partial\varphi(y)=\{z\in Y' : \varphi(y)-\varphi(u)\le \langle y-u,z \rangle, \forall y\in Y\}.
\end{equation}
The function $\varphi^\ast\colon Y'\to \bar \R$ defined by
\begin{equation}\label{e1.7}
\varphi^\ast(z)=\sup\{ \langle y,z \rangle-\varphi(y) : y\in Y\}
\end{equation}
is called the \textbf{conjugate} of $\varphi$ and likewise $\varphi$ it is convex and lower-semicontinuous on $Y'$.
Also we notice the following key conjugacy formulae (see e.g., \cite[pag. 6]{B0})
\begin{equation}\label{e1.8}
\varphi(y)+\varphi^\ast(z)=\langle y,z\rangle\quad \mbox{iff } z\in\partial\varphi(y)
\end{equation}
\begin{equation}\label{e1.9}
\varphi(y)+\varphi^\ast(\bar z)\ge\langle y,\bar z\rangle\quad \forall y\in Y,\bar z\in Y'
\end{equation}
(Here $\langle\cdot,\cdot\rangle$ is the duality pairing between $Y$ and $Y'$.)}

\section{The main results}
We shall study here equation \eqref{e1.2} under the following assumptions\\[6pt]
\textbf{(i)} \emph{ $V$ is a separable real reflexive Banach space with the dual $V'$ and $H$ is a separable real Hilbert
space such that $V\subset H\subset V'$ algebraically and topologically.
}\\[6pt]
We denote by $|\cdot|$, $\|\cdot\|_V$ and $\|\cdot\|_{V'}$ the norms in $H$, $V$ and $V'$, respectively and
by $\langle\cdot,\cdot\rangle$ the duality pairing on $V\times V'$ which coincides with the scalar
product  $(\cdot,\cdot)$ of $H$ on $H\times H$.\\[6pt]
\textbf{(ii)} \emph{$A(t)y=\partial\psi(t,y)$, a.e. $t\in(0,T)$, $\forall y\in V$, $\P$-a.s.,
where $\psi\colon (0,T)\times V\times\Omega\to\R$ is convex and lower-semicontinuous in $y$ on $V$
and measurable in $t$ on $[0,T]$. There are $\alpha_i>0$, $\gamma_i\in\R$, $i=1,2$, $1<p_1\le p_2<\infty$
\begin{equation}\label{e1.10}
\gamma_1+\alpha_1\,\|y\|_V^{p_1}\le \psi(t,y)\le \gamma_2+\alpha_2\,\|y\|_V^{p_2},\quad \forall y\in V.
\end{equation}}\\[6pt]
\textbf{(iii)} \emph{There are $C_1$, $C_2\in \R^+$ such that
\begin{equation}\label{e1.11}
 \psi(t,-y)\le C_1\,\psi(t,y)+C_2,\qquad \forall y\in V, t\in
 (0,T).
\end{equation}
}\\
(The constants $C_i$, $\gamma_i$, $\alpha_i$ are dependent on $\omega$.)\\[6pt]
\textbf{(iv)} \emph{For each $y\in V$ the stochastic process $\psi(t,y)$ is progressively measurable with respect to
filtration $\{\mathcal F_t\}_{t\ge0}$.}\\[6pt]
\textbf{(v)} \emph{$B_j(t)$ is a family of linear, closed and
densely defined operators in $H$ such that $B_j(t)=-B_j^*(t)$,
$\forall t\in [0,T]$, $B_j(t)$ generates a $C_0$-group
$(e^{s\,B_j(t)})_{s\in\RR}$ on $H$ and $V$. Moreover,
$B_j\in C^1([0,T];L(V,H))$, $B_j(t)B_k(t)=B_k(t)B_j(t)$ for all $j$, $k$.}\\[6pt]
\textbf{(vi)} \emph{$f\colon [0,T]\times\Omega\to V'$ is progressively measurable and $f\in L^{p_1'}(0,T;V')$,
$\P$-a.s.
}\\[6pt]
We note that by (ii) $A(t,\omega)\colon V\mapsto V'$ is, for all $t\in[0,T]$ and $\omega\in\Omega$,
maximal monotone and surjective (see \cite{barbu-2010}) but in general multi-valued if $\psi$ is not G\^ateaux differentiable in $y$.
\begin{Theorem}\label{t2.1}
Let $y_0\in H$. Then under assumptions {\bf(i)} $\sim$ {\bf(vi)} there is for each $\omega\in\Omega$
a unique function $y=y(t,\omega)$ to equation \eqref{e1.2} which satisfies
\begin{equation}\label{e1.12}
y\in L^{p_1}(0,T;V)\cap C([0,T];H)\cap W^{1,p_2'}([0,T];V'),
\end{equation}
\begin{equation}\label{e1.13}
\lk\{
\begin{array}{l}
\ds\frac{dy}{dt}(t)+\Lambda(t)y(t)\ni g(t), \quad \mbox{a.e. } t\in (0,T),\\\ \\
y(0)=y_0.
\end{array}
\rk.
\end{equation} Moreover, the process $y\colon
[0,T]\times\Omega\to H$ is progressively measurable with respect
to the filtration $\{\mathcal F_t\}_{t\ge0}$.
\end{Theorem}
If $A_n(t,\cdot)$ and $A_n^{-1}(t,\cdot)$ are multi-valued we mean by $A_n(t,y)$ and $A_n^{-1}(t,z)$
single valued sections.\\
By $\stackrel{\scriptstyle G}{\to}$ we denote the variational
or $\Gamma$-convergence. This means that for each $y\in V$ and $\xi\in A(t)y$ there are $y_n$
and $\xi_n\in A(t)y$ such that $y_n\to y$ strongly in $V$, $\xi_n\to \xi$ strongly in $V'$
and similarly for $A_n^{-1}(t)\to A^{-1}(t)$. Assumption \eqref{e1.14} implies
and is equivalent to: $\psi_n(t,z)\to \psi(t,z)$, $\psi_n^*(t,\tilde z)\to \psi^*(t,\tilde z)$
for all $z\in V$, $\tilde z\in V'$ and $t\in [0,T]$ where $\psi^*$ is the conjugate of $\psi$ (See e.g., \cite{ABM}.)
\vspace{12pt}

\begin{Theorem}\label{t2.2}
Assume that for each $n$, $\Lambda_n$, $B_j^n$ and $f_n$ satisfy
{\bf (i)} $\sim$ {\bf (iv)}. Then for any $y_0\in V$ there is a
unique function $y_n=y_n(t,\omega)$ which satisfies \eqref{e1.12}
and equation \eqref{e1.13} with $\Lambda_n$ instead of $\Lambda$.
Moreover, assume that for $n\to\infty$
\begin{equation}\label{e1.14}
\begin{array}{l}
A_n(t)\;\stackrel{\scriptstyle G}{\to}\; A(t),\qquad \;\;\; t\in [0,T]
\\
A_n^{-1}(t)\stackrel{\scriptstyle G}{\to} A^{-1}(t),\qquad t\in [0,T]
\end{array}\end{equation}
\begin{equation}\label{e1.15}
f_n(\cdot,\omega) \to f(\cdot,\omega),\quad \mbox{strongly in }L^{p_2'}(0,T;V'), \P\mbox{-a.s. in }\Omega.
\end{equation}
\begin{equation}\label{e1.15a}
B_j^nx  \to B_jx,\quad \mbox{in }C^1([0,T];H), \qquad \forall x\in H.
\end{equation}
Then for $n\to \infty$
\begin{equation}
y_n(\cdot,\omega)\to y(\cdot,\omega)
\end{equation}
$ \P\mbox{-a.s. weakly in }L^{p_1}(0,T;V),\mbox{ weakly-star in } L^\infty(0,T;H).$
\end{Theorem}
Assumption \eqref{e1.14} implies
and is equivalent to: $\psi_n(t,z)\to \psi(t,z)$, $\psi_n^*(t,\tilde z)\to \psi^*(t,\tilde z)$
for all $z\in V$, $\tilde z\in V'$ and $t\in [0,T]$ where $\psi^*$ is the conjugate of $\psi$ (See e.g., \cite{ABM}.)\\
Coming back to equation \eqref{e1.1} we say that the process
$X:[0,T]\to H$ is a solution to \eqref{e1.1}, if it is
progressively measurable with respect to the filtration
$\{\mathcal F_t\}_{t\ge 0}$ induced by the Brownian motion, \DEQSZ
\label{eq2.8} X\in C([0,T],H) \cap L ^{p_1}(0,T;V),\quad \mbox{
$\PP$--a.s.} \EEQSZ and \DEQSZ \label{eq2.9}\hspace{25pt}X(t) &=
&x-\int_0 ^t A(s) X(s) \, ds +\frac 12 \sum_{j=1}^N\int_0 ^t B_j^2 (s) X(s)\,
ds \\ &&{}+ \sum_{j=1}^N\int_0 ^t B_j(s) X(s) d\beta_j(s) + \int_0 ^t f(s) \,
ds,\quad \forall t\in[0,T],\, \mbox{ $\PP$--a.s.}. \nonumber\EEQSZ
By Theorem \ref{t2.1} and Theorem \ref{t2.2} we find that
\begin{Theorem}\label{t2.3}
Under the assumptions of Theorem \ref{t2.1} there exist unique
solutions $X$ and $X_n$ to \eqref{e1.1} and \eqref{e1.5} respectively given by
\begin{equation}
X(t)= e^{\sum_{j=1}^N\beta_j(t)B_j(t)}y(t),\qquad
X_n(t)=e^{\sum_{j=1}^N\beta_j^n(t)B_j^n(t)}y_n(t),
\end{equation}
where $y$ and $y_n$ are solutions to \eqref{e1.2} and
\eqref{e1.2}$_n$. Moreover, we have
\begin{equation}
X, X_n \in L ^{p_1}(0,T;V),\qquad  \P\mbox{-a.s.}
\end{equation}
$X, X_n\colon [0,T]\to H$ are $\P$-a.s. continuous and
\begin{equation}
X_n\to X\quad\mbox{weakly in } L^{p_1}(0,T;V), \mbox{weakly-star in } L^{\infty}(0,T;H),\ \ \P\mbox{-a.s.}
\end{equation}
\end{Theorem}
The precise meaning of Theorem \ref{t2.2} and Theorem \ref{t2.3}
is the structural stability of the It\^o stochastic differential
equation \eqref{e1.1}
and of its Stratonovich counterpart (1.1)'. As a matter of fact, as mentioned earlier,
all these results can be reformulated in terms of Stratonovich equation (1.1)'. \\
One of the main consequences of Theorem \ref{t2.2} is that the
Stratonovich stochastic equation is stable with respect to smooth
approximations of  the process $B(t)Xd\beta(t)$. On the other
hand, the general existence theory for infinite dimensional
stochastic differential equations with linear multiplicative noise
(see e.g., \cite{DaP,DZ}) is not applicable in present situation
due to the fact that the noise coefficient $x\to B(t)x$ is not
bounded on the basic space $H$. The approach we use here to treat equation \eqref{e1.13}
relies on Brezis-Ekeland variational principle \cite{BrEk0}, \cite{BrEk} which allows to reduce nonlinear
evolution equations of potential type to convex optimization problems. (On these lines see also
\cite{B}-\cite{barbu-1987}, \cite{rr}, \cite{vis}).\\
The more general case of nonlinear monotone and demicontinuous operators
$A(t)\colon V\to V'$ is ruled out from present theory and might expect however to extend the theory
to this general case by using Fitzpatrick function formalism (see \cite{rr}, \cite{vis}).

\bigskip

As in \cite{brz-cap-flan}, see Corollary on p. 438, we can define a solution to problem $(1.1)'$ for any \textbf{deterministic} continuous  function $\beta:\mathbb{R}_+\to \mathbb{R}^N$. The result from \cite{brz-cap-flan} was a generalisation of a analogous result from Sussmann's well known paper \cite{Sussmann_1978}, see also Doss \cite{Doss_1977}. We will formulate our result in the same fashion as in \cite{brz-cap-flan}, i.e. the result contains implicitly a definition. Let us observe, that in this case, we prove the existence for any multidimensional continuous signal, thus a signal more general than a rough signal from the theory of rough paths. However, this is due to the assumption of the commutativity of the vector fields $B_j$, $j=1,\cdots,N$.  In the result below, we need deterministic versions of assumptions {\bf(ii)},  {\bf(vi)}. Note that the assumption {\bf(iv)} is now redundant.

\textbf{(ii')} \emph{$A(t)y=\partial\psi(t,y)$, a.e. $t\in(0,T)$, $\forall y\in V$, where $\psi\colon (0,T)\times V\to\R$ is convex and lower-semicontinuous in $y$ on $V$ and measurable in $t$ on $[0,T]$. There exist $\alpha_i>0$, $\gamma_i\in\R$, $i=1,2$, $1<p_1\le p_2<\infty$, such that
\begin{equation}\label{e1.10z}
\gamma_1+\alpha_1\,\|y\|_V^{p_1}\le \psi(t,y)\le \gamma_2+\alpha_2\,\|y\|_V^{p_2},\quad \forall y\in V.
\end{equation}}\\[6pt]
\textbf{(vi)}  $f\in L^{p_1'}(0,T;V^\prime)$,

\begin{Theorem}\label{thm-sussmann} Assume that the  assumptions {\bf(i)}, {\bf(iii)}, {\bf(v)} as well as {\bf(ii')} and  {\bf(vi')}   are satisfied. Then for every $x\in V$ and every $\beta \in C([0,T];\mathbb{R}^N)$, the problem
\begin{equation}\label{e2.14}\lk\{
\begin{array}{l}
dX(t)+A(t)X(t)\,dt =\ds\sum_{j=1}^N B_j(t)X(t)\,d\beta_j(t)+f(t)\,dt,\quad t\in[0,T]
\\ \\
X(0)=x
\end{array}\rk.
\end{equation}
has a unique solution $X\in L ^{p_1}(0,T;V)\cap  C([0,T];H)$ in the following sense. 
\begin{trivlist}
\item[(i)] For every $\beta \in C^1([0,T];\mathbb{R}^N)$, the problem \eqref{e2.14} has a unique solution $X\in L ^{p_1}(0,T;V)\cap  C([0,T];H)$.
\item[(ii)] If $\beta_n \in C^1([0,T];\mathbb{R}^N)$ and $\beta_n \to \beta$ in $C([0,T];\mathbb{R}^N)$ and $X_n \in L ^{p_1}(0,T;V)\cap  C([0,T];H)$ is the (unique) solution to
 the problem \eqref{e2.14} corresponding to $\beta_n$, then $X_n \to X$ in $\mbox{weakly in } L^{p_1}(0,T;V), \mbox{weakly-star in } L^{\infty}(0,T;H),\ \ \P\mbox{-a.s.}$
\end{trivlist}
\end{Theorem}

From Theorem \ref{thm-sussmann} we infer that in the framework of Theorem \ref{thm-sussmann} but with $\beta$ being a Brownian motion, the problem \eqref{e1.1} generates  a random dynamical system on $H$. In an obvious way we have the following Corollary

\begin{Corollary}\label{cor-RDS} Assume that the  assumptions {\bf(i)}, {\bf(iii)}, {\bf(v)} as well as {\bf(ii')} and  {\bf(vi')}   are satisfied.
Assume that $\beta$ in a standard canonical two-sided $\mathbb{R}^N$-valued  Brownian
motion on  a filtered probability space $\{\Omega,\mathcal F,\{\mathcal F_t\}_{t\ge 0},\P\}$, where $\Omega=\{ \omega \in C(\mathbb{R},\mathbb{R}^N): \omega(0)=0\}$. Let us define a map
 \[\vartheta:\mathbb{R}\times\Omega\ni (t,\omega)\mapsto
\vartheta_{t}\omega=\omega(\cdot +t) -\omega(0)\in \Omega. \]
 Then there exists a map
\[\varphi: \mathbb{R}^+\times\Omega\times H \ni (t,\omega,x) \mapsto
\varphi(t,\omega)x \in H\]
such that a pair $(\phi,\vartheta)$ is a random dynamical system on $H$, see for instance Definition 2.1 in \cite{Brz+Li_2006}, and, for each $s\in\mathbb{R}$ and each $x\in H$,  the process $X$, defined for $\omega \in \Omega$ and $t \geq s$ as
\begin{equation}\label{eqn:4.12}
X(t,s;\omega,x):=\varphi(t-s;\vartheta_{s}\omega)x,
\end{equation}
 is a solution to problem \eqref{e1.1} over the time interval $[s,\infty)$ with an initial data given at time $s$.
\end{Corollary}
\begin{Remark}\label{rem-sussmann}\em
Theorem \ref{thm-sussmann} provides a solution to problem \eqref{e2.14} for every continuous path. Our main result provides a natural interpretation of this solution in the case when $\beta$ is a Brownian motion. One can also provide a similar interpretation when $\beta$ is a fractional, see for instance \cite{DMP_2005}. \\
Corollary allows one to investigate the existence of random attractors, see \cite{Brz+Li_2006}. \\
These questions will be investigated in the future works.
\end{Remark}

 \section{PROOFS}
\emph{Proof of Theorem 2.1}\\
For simplicity we consider the case $N=1$, that is $B_j=B$, $\beta_j=\beta$ for all $j$.\\
We note first that though the operator
$\Lambda(t)=\Lambda(t,\omega)$ is $\P$-a.e $\omega\in\Omega$,
maximal monotone from $V$ to $V'$ the standard existence theory
(see e.g.,  \cite[pag. 177]{barbu-2010}) does not apply here. This
is, however, due to the general growth condition \eqref{e1.10} on $\psi(t,\cdot)$
and implicitly on $A(t)$ as well as due to the multivaluedness of $A(t)$.\\
So we shall use a direct approach which makes use of the
variational structure of problem \eqref{e1.2}. (On these lines see
also \cite{B}, \cite[page 280]{barbu-1987}). Namely, we can write
\begin{equation}
\Lambda(t)=\partial\varphi(t,\cdot)+\Gamma(t),\quad \forall t\in [0,T].
\end{equation}
Here $\varphi\colon [0,T]\times V\to\R$ is given by
\begin{equation}\label{e3.2}
\varphi(t,y)=\psi(t,e^{\beta(t)B(t)}y)
\end{equation}
and
\[
\Gamma(t)y=\int_0^{\beta(t)}e^{-s\,B(t)}\dot B(t)e^{s\,B(t)}y\,ds,\qquad \forall y\in H, t\in [0,T]
\]
where $\dot B=\frac{d\ }{dt}B(t)$. We fix $\omega\in\Omega$.\\
By the conjugacy formulae \eqref{eq-1.9} and  \eqref{eq-1.8} we
now may equivalently write \eqref{e1.2} (or \eqref{e1.13}) as
\begin{equation}\lk\{
\begin{array}{l}
\varphi(t,y(t))+\varphi^\ast(t,u(t))=\langle y(t), u(t)\rangle, \qquad\mbox{a.e. } t\in [0,T]\\
y'(t)+\Gamma(t)y(t)=-u(t)+g(t), \qquad\mbox{a.e. } t\in [0,T]
\end{array}\rk.
\end{equation}
while
\[
\varphi(t,\bar y)+\varphi^\ast(t,\bar u)\ge \langle \bar y,\bar u\rangle
\]
for all $\langle \bar y,\bar u\rangle\in L^{p_1}(0,T;V)\times L^{p_2'}(0,T;V')$.

Thus following a well known idea due to Brezis and Ekeland (see
e.g., \cite{,BrEk0,BrEk}) we are lead to the optimization problem
\begin{multline}\label{e3.4}
\mbox{Min}\{\int_0^T \big(\varphi(t,y(t))+\varphi^\ast(t,u(t))-
\langle u(t),y(t)\rangle\big)\,dt : \\y'+\Gamma(t)y=-u+g, \mbox{
a.e. } t\in [0,T] ; \, y(0)=y_0, \\ y\in L^{p_1}(0,T;V), u\in
L^{p_2'}(0,T;V')\}.
\end{multline}
Equivalently
\begin{multline}\label{e3.5}
\mbox{Min}\{\int_0^T \big(\varphi(t,y(t))+\varphi^\ast(t,u(t))-
\langle g(t),y(t)\rangle\big)\,dt
+\tfrac12\big(|y(T)|^2-|y_0|^2\big): \\y'+\Gamma(t)y=-u+g, \mbox{
a.e. } t\in [0,T] ; \, y(0)=y_0, \\ y\in L^{p_1}(0,T;V), u\in
L^{p_2'}(0,T;V')\}.
\end{multline}
Here we have used (for the moment, formally) the integration by parts formula
\begin{multline*}
-\int_0^T\langle u(t),y(t)\rangle\,dt=\tfrac12\big(|y(T)|^2-|y_0|^2\big)\\
+\int_0^T \langle\Gamma(t)y(t),y(t)\rangle\,dt-\int_0^T\langle g(t),u(t)\rangle\,dt
\end{multline*}
and hypothesis (v) which implies that
$\langle\Gamma(t)y,y\rangle=0$. Of course the equivalence of
\eqref{e3.4} and \eqref{e3.5} is valid only if the above equality
is true which is not always the case in absence of some additional properties
of  minimizer $y$ to allow integration by parts in $\int_0^T\langle u(t),g(t)\rangle\,dt$.
In the following we shall prove however that equation
\eqref{e3.5} has at least one solution and show consequently that
it is also a solution to equation \eqref{e1.13}.
\begin{Lemma}\label{L3.1}
There is a solution $y^\ast\in L^{p_1}(0,T;V)\cap
W^{1,p_2'}([0,T];V')$ to equation \eqref{e3.5}.
\end{Lemma}
\begin{proof}
We note that by standard existence theory of linear evolution equations for each $u\in L^{p_2'}(0,T;V')$
there is a unique solution $y\in L^{p_2}(0,T;V)\cap W^{1,p_2'}([0,T];V')\subset C([0,T];H)$
to equation
\[
y'+\Gamma(t,y)=-u+g,\quad\mbox{a.e. } t\in[0,T],\quad y(0)=y_0.
\]
By assumptions \eqref{e1.10} and \eqref{e1.11} we have
\begin{equation}
\tilde\gamma_1+\tilde\alpha_1\,\|y\|_V^{p_1}\le \varphi(t,y)\le
\tilde\gamma_2+\tilde\alpha_2\,\|y\|_{V}^{p_2},\qquad \forall y\in V
\end{equation}
and
\begin{equation}
\bar\gamma_1+\bar\alpha_1\,\|y\|_{V'}^{p_2'}\le \varphi^\ast(t,y)\le
\bar\gamma_2+\bar\alpha_1\,\|y\|_{V'}^{p_1'},\qquad \forall y\in V'
\end{equation}
where $\tilde\gamma_i,\bar\gamma_i\in\R$, $\frac1{p_i}+\frac1{p_i'}=1$ and
$\tilde\alpha_i,\bar\alpha_i>0$, $i=1,2$. (Recall that $e^{s B(t)}$ is invertible).\\
Then the infimum $m^\ast$ in \eqref{e3.5} is $>-\infty$ and there are the sequences $\{y_j\}\subset L^{p_1}(0,T;V)$,
$\{u_j\}\subset L^{p_2'}(0,T;V')$ such that for all $y$
\begin{equation}\label{e3.8}
m^\ast\le \int_0^T\big(\langle\varphi(t,y_j)+\varphi^\ast(t,u_j)-\langle g,y_j\rangle\big)\,dt+\tfrac12\big(|y_j(T)|^2-|y_0|^2\big)\le m^\ast+\frac1j
\end{equation}
\begin{equation}\lk\{
\begin{array}{l}
y_j'+\Gamma(t)y_j=-u_j+g,\qquad \mbox{ in } [0,T]\\
y_j(0)=y_0.
\end{array}
\rk. 
\end{equation} Clearly $y_j\in W^{1,p_2'}([0,T];V')$ and by
assumption  \eqref{e1.10} and inequality  \eqref{e3.8} it follows
that
\begin{equation}
\|y_j\|_{L^{p_1}(0,T;V)}+\|y_j'\|_{L^{p_2'}(0,T;V')}\le C
\end{equation}
because as easily seen by assumption ({\bf v}), $|\Gamma(t)y|_H\le C \|y\|_{V}$, $\forall y\in V$.\\
Hence on a subsequence, again denoted $y_j$, we have for $j\to\infty$
\begin{equation}
\begin{array}{l}
y_j\to y^\ast\quad \mbox{weakly in } L^{p_1}(0,T;V)\\
u_j\to u^\ast\quad \mbox{weakly in } L^{p_2'}(0,T;V')\\
y_j'\to (y^\ast)'=g-u^\ast-\Gamma(t)y^\ast\quad \mbox{weakly in } L^{p_2'}(0,T;V').
\end{array}
\end{equation}
Since the functions $y\mapsto\int_0^T\varphi(t,y)\,dt$ and $u\mapsto\int_0^T\varphi^\ast(t,u)\,dt$ are weakly lower-semicontinuous
on $L^{p_1}(0,T;V)$ and $L^{p_2'}(0,T;V')$ respectively, letting $j$ tend to infinity we obtain that
\begin{equation}\label{e3.12}
m^\ast=\int_0^T\big(\varphi(t,y^\ast)+\varphi^\ast(t,u^\ast)-\langle g,y^\ast\rangle\big)\,dt+\tfrac12\big(|y^\ast(T)|^2-|y_0|^2\big)
\end{equation}
and
\begin{equation}\label{e3.13}\lk\{
\begin{array}{l}
(y^\ast)'+\Gamma(t)y^\ast=-u^\ast+g,\qquad t\in [0,T]\\
y^\ast(0)=y_0.
\end{array}\rk.
\end{equation}
Therefore $(y^*,u^*)$ is a solution to optimization problem
\eqref{e3.5} as claimed.
\end{proof}
\begin{proof}[Proof of Theorem 2.1 (continued)]
We shall show now that $y^\ast$ given by Lemma \ref{L3.1} is a solution to \eqref{e1.13}.
To this end we notice just that without any loss of generality we may assume that $y_0=0$. Indeed
we can reduce the problem to this case by translating in problem \ref{e1.12} $y$ in $y-y_0$.

We prove now that $m^\ast=0$.
For this purpose we invoke a standard duality result for infinite dimensional convex optimal control problems,
essentially due to R.T. Rockafeller. Namely, one has (see \cite[Thm. 4.6, pag. 287]{barbu-1987})
\begin{equation}\label{e3.14}
m^\ast+\min\eqref{e3.5}'=0
\end{equation}
where \eqref{e3.5}$'$ is the dual control problem
$$
\leqno{(3.5)'}\begin{aligned}\qquad \mbox{Min}\{\int_0^T\big(\varphi(t,-p(t))+\varphi^\ast(t,v(t)) +\langle g(t),p(t)\rangle\big)\, dt+\tfrac12|p(T)|^2:\\
p'+\Gamma(t)p=v+g,\qquad t\in [0,T]\}.\end{aligned}
$$
If $(p^\ast,v^\ast)\in L^{p_1}(0,T;V)\times L^{p_2'}(0,T;V')$ is
optimal in (3.5)$'$, we have
\begin{equation}\label{e3.15}
\langle(p^\ast)',p^\ast\rangle\in L^{1}(0,T)
\end{equation}
\begin{equation}\label{e3.16}
\int_0^T\langle(p^\ast)',p^\ast\rangle\, dt=\tfrac12\big(|p^\ast(T)|^2-|p^\ast(0)|^2\big).
\end{equation}
Here is the argument. First, note that $p'$ solves
$p'+\Gamma(t)p=v+g$. We have by the identities \eqref{eq-1.9} and
\eqref{eq-1.8} and the fact that $\langle
\Gamma(t)p^\ast,p^\ast\rangle=0$,
\[
-\langle(p^\ast(t))',p^\ast(t)\rangle\le \varphi^\ast(t,v^\ast(t))+\varphi(t,-p^\ast(t))
-\langle g(t),p^\ast(t)\rangle,\qquad \mbox{a.e. } t\in (0,T)
\]
and
\[
\langle(p^\ast(t))',p^\ast(t)\rangle\le\varphi^\ast(t,v^\ast(t))+\varphi(t,p^\ast(t))
+\langle g(t),p^\ast(t)\rangle,\qquad \mbox{a.e. } t\in (0,T).
\]
Since $\varphi(t,-p^\ast)\in L^1(0,T)$ and by assumption
\eqref{e1.11}, $\varphi(t,p^\ast)\in L^1(0,T)$ too, we infer that
\eqref{e3.15} holds. Now since $p^\ast\in W^{1,p_2'}([0,T];V')\cap
L^{p_1}(0,T;V')$ we have
\[
\tfrac12 \frac{d\ }{dt}|p^\ast(t)|^2=\langle (p^\ast)'(t),p^\ast(t)\rangle, \qquad \mbox{a.e. } t\in (0,T)
\]
and by \eqref{e3.15} we get \eqref{e3.16} as claimed.\\
By (3.5)$'$ and \eqref{e3.16} we see that
\[
\min \;(3.5)'=\int_0^T\big(\varphi(t,-p^\ast)+\varphi^\ast(t,v^\ast)+
\langle v^\ast,p^\ast\rangle\big)\,dt+\tfrac12 |p^\ast(0)|^2\ge 0.
\]
Similarly, by \eqref{e3.12}, \eqref{e3.13} and by
\[
\tfrac12\big(|y^\ast(T)|^2-|y^\ast(0)|^2\big)=\int_0^T\langle
(y^\ast)',y^\ast\rangle\,dt,
\]
(the latter follows exactly as \eqref{e3.16}) we see that
\[
m^\ast=\int_0^T\big(\varphi(t,y^\ast)+\varphi^\ast(t,u^\ast)-\langle u^\ast,y^\ast \rangle \big)\,dt\ge 0.
\]
Then by \eqref{e3.14} we have that $m^\ast=0$ and therefore again
by \eqref{e3.12} we have that
\[
\begin{array}{l}
\varphi(t,y^\ast)+\varphi^\ast(t,u^\ast)=\langle u^\ast,y^\ast
\rangle,\qquad \mbox{a.e.\ in}\quad  [0,T]
\\
(y^\ast)'+\Gamma(t)y^\ast=g-u^\ast\qquad\mbox{a.e.\ in}\quad [0,T]
\end{array}
\]
and therefore $y^\ast$ is a solution to \eqref{e1.13}.\\
On the other hand, as seen earlier, we have
\begin{equation}\label{e3.17}
\tfrac12 \big(|y^\ast(t)|^2-|y^\ast(s)|^2\big)=\int_s^t \langle
(y^\ast)'(\tau),y^\ast(\tau)\rangle\,d\tau,\quad \forall \, 0\le
s\le t\le T.
\end{equation}
Hence $y^\ast\in C([0,T];H)$. The uniqueness of $y^\ast$ is immediate by \eqref{e3.17}. It remains to
be proven that $y^\ast$ is progressively measurable.\\
To this end we note as minimum in \eqref{e3.5} the pair $(y^\ast,u^\ast)$
is the solution to Euler-Lagrange system (see e.g., \cite[page 263]{barbu-1987})
\[
\begin{array}{l}
(y^\ast)'+\Gamma(t)y^\ast=-u^\ast+g,\qquad \mbox{a.e. } t\in (0,T), \ \omega\in\Omega\\
q'-\Gamma'(t)q= -g+A(t)y^\ast,\qquad \mbox{a.e. } t\in (0,T)\\
u^\ast(t)=A(t)q(t),\qquad \mbox{a.e. } t\in (0,T), \ \omega\in\Omega\\
y^\ast(0)=y_0,\quad q(T)=-y^\ast(T).
\end{array}
\]
Since the latter two point boundary value problem has a unique
solution $(y^\ast,q)$ and is of dissipative (accretive) type it
can be solved by iteration or more precisely by a gradient
algorithm (see \cite[page 252]{barbu-1987}). In particular, we
have  $y^\ast=\lim_{k\to\infty} y_k$, $q=\lim_{k\to\infty} q_k$ weakly in $L^{p_1}(0,T;V)$ and
$u^\ast=\lim_{k\to\infty} u_k$  weakly in $L^{p_2'}(0,T;V')$ where
\[
\begin{array}{l}
y_k'+\Gamma(t)y_k=-u_k+g\qquad  t\in [0,T], \\
q_k'-\Gamma'(t)q_k= -g+A(t)y_k,\qquad  t\in [0,T],\\
u_{k+1}=u_k-A^{-1}(t)u_k+q_k,\qquad  t\in [0,T],\\
y_k(0)=y_0,\quad q_k(T)=0,\quad k=0,1,2,\ldots.
\end{array}
\]
Hence, if we start with a progressively measurable $u_0$, we see that all $u_k$ are progressively measurable and so
are $u^\ast$ and $y^\ast$.\end{proof}
\bigskip
\begin{proof}[Proof of Theorem 2.2] As in the previous case it follows that equation \eqref{e1.2}$_n$
has a unique solution $y_n\in W^{1,p_2'}([0,T];V')\cap
L^{p_1}(0,T;V)$ given by the minimization problem \eqref{e3.5}
where $g=g_n$ and $\varphi=\varphi_n$,
$\varphi^\ast=\varphi_n^\ast$. Here $\varphi_n$ is given as in
\eqref{e3.2} where $\psi=\psi_n$ and $\beta$ is replaced by
$\beta_n$ while $\varphi_n^\ast$ is the conjugate of $\varphi_n$.
We have, similarly, $\partial\psi_n=A_n$ and
$\varphi_n(t,y)=\psi_n(t,e^{\beta_n(t)y})$
\begin{multline*}
(y_n,u_n)=\mbox{arg min}\big\{\int_0^T\varphi_n(t,y(t))+\varphi^*_n(t,u(t))-\langle g_n(t),u(t)\rangle\,dt\\
+\tfrac12 (|y(T)|^2-|y_0|^2);\qquad y'+\Gamma_n y=-u+g_n,\quad
y(0)=y_0\big\}.
\end{multline*}
Here $\Gamma_n(t)y=\int_0^{\beta_n(t)} e^{-s B_n(t)}\dot B_n(t) e^{sB_n(t)}y\;ds$.\\
We see that
\[
\|y_n\|_{L^{\infty}(0,T;H)}+\|y_n\|_{L^{p_1}(0,T;V)}+\Big\|\frac{dy_n}{dt}
\Big\|_{L^{p_2'}(0,T;V')}+|y_n(T)|\le C,
\]
and this implies that on a subsequence, again denoted $\{n\}$, we have
\addtocounter{equation}{1}
\begin{equation}\label{e3.19}\lk\{
\begin{array}{l}
u_n\to \tilde u\quad \mbox{weakly in } L^{p_2'}(0,T;V')\\ \\
y_n \longrightarrow \tilde y \quad\mbox{weakly in } L^{p_1}(0,T;V)\\ \\
y_n \longrightarrow \tilde y \quad\mbox{weakly-star in } L^{\infty}(0,T;H)\\ \\
\ds\frac{dy_n}{dt}  \longrightarrow \frac{d\tilde y}{dt} \quad\mbox{weakly in }  L^{p_2'}(0,T;V')\\ \\
y_n(T)  \longrightarrow \tilde y(T) \quad\mbox{weakly in } H.
\end{array}\rk.
\end{equation}
By \eqref{e1.14}, \eqref{e1.15}  we see that $\tilde y'+\Gamma \tilde y=-\tilde u+g$.\\
Moreover, we have
\[
\begin{split}
\int_0^T\big(\varphi_n(t,y_n(t))+\varphi_n^\ast(t,u_n(t))-\langle g_n(t),y_n(t)\rangle\big)\,dt +\tfrac12 |y_n(T)|^2\\
\le \int_0^T\big(\varphi_n(t,y^\ast(t))+\varphi_n^\ast(t,u^\ast(t))-\langle g_n(t),u^\ast(t)\rangle\big)\,dt +\tfrac12 |y^\ast(T)|^2.
\end{split}
\]
where $(y^\ast,u^\ast)$ is the solution to \eqref{e3.5}.
Now by assumptions \eqref{e1.14} and  \eqref{e1.15} we have
\[
\begin{array}{l}
\varphi_n(t,y^\ast(t))\to \varphi(t,y^\ast(t))\\ \\ \varphi_n^\ast(t,u_n^\ast(t))\to \varphi(t,u^\ast(t))\\ \\
g_n(t)\to g(t)\end{array}\qquad \forall t\in[0,T],
\]
and this yields
\begin{equation}\label{e3.20}
\limsup_{n\to\infty} \int_0^T\big(\varphi_n(t,y_n)+\varphi_n^\ast(t,u_n)-\langle g_n,y_n\rangle\big)\,dt +\tfrac12 |y_n^\ast(T)|^2
-\tfrac12 |y_0|^2=0.
\end{equation}
In order to pass to limit in \eqref{e3.20} we shall use \eqref{e3.19}
and the convergence of $\{\varphi_n\}$ and $\{\varphi_n^\ast\}$ mentioned above.
We set $\tilde z(t)=e^{\beta(t)\,B(t)}\tilde y(t)$, $z_n(t)=e^{\beta_n(t)\,B(t)}y_n(t)$. We have
\[
\psi_n(t,\tilde z(t))\le\psi_n(t,z_n(t))+\langle A_n(t,\tilde
z(t)),\tilde z(t)-z_n(t)\rangle,
\]
and since $\partial\psi_n^\ast=A_n^{-1}$ we have also that
\[
\psi_n^\ast(t,\theta(t))\le \psi_n^\ast(t,\theta_n(t))
+\langle A_n^{-1}(t)(t,\theta(t)),\theta(t)-\theta_n(t)\rangle,\quad \mbox{a.e. }t\in (0,T).
\]
where $\theta(t)=g(t)-\tilde y'(t)$, $\theta_n(t)=g_n(t)- y_n'(t)$.\\
Then by assumption \eqref{e1.14} and equation \eqref{e3.20} we have that
\begin{multline*}
\limsup_{n\to\infty} \int_0^T\big(\varphi_n(t,\tilde y(t))+\varphi_n^\ast(t,g(t)-\tilde y'(t))-\\
-\langle g(t),\tilde y'(t)\rangle\big)\,dt +\tfrac12 |\tilde y^\ast(T)|^2-\tfrac12 |y_0|^2=0,
\end{multline*}
and so, since as seen earlier \eqref{e1.14} implies that $\varphi_n(t,z)\to
\varphi(t,z)$, $\forall z\in V$,
$\varphi_n^\ast(t,z^\ast)\to\varphi^\ast(t,z^\ast)$, $\forall
z^\ast\in V'$, by the Fatou lemma, we have
\[
 \int_0^T\big(\varphi(t,\tilde y)+\varphi^\ast(t,\tilde u)
-\langle g,\tilde y\rangle\big)\,dt +\tfrac12 |\tilde
y^\ast(T)|^2-\tfrac12 |y_0|^2\le 0,
\]
which implies as in the previous case that $\tilde y$ is a solution to \eqref{e3.5}
and therefore to \eqref{e1.13} as claimed.
\end{proof}

\section{Examples}
The specific examples to be presented below refer to nonlinear parabolic
stochastic equations which can be written in the abstract form \eqref{e1.1} where $A(t)$ are
subpotential monotone and continuous operators from a separable Banach space $V$
to its dual $V'$.\\
We briefly present below a few stochastic PDE  to which the above theorems apply.
We use here the standard notations for spaces of integrable functions and Sobolev spaces
$W_0^{1,p}(\mathcal O), W^{-1,p'}(\mathcal O)=\big(W_0^{1,p}(\mathcal O)\big)', H^k(\mathcal O)$, $k=1,2$ on open domains
$\mathcal O\subset \R^d$.

\subsection{Nonlinear stochastic diffusion equations}
Consider the stochastic equation
\begin{equation}\label{e4.1}
\lk\{\begin{array}{l}
\begin{aligned}dX_t-\mbox{div}_\xi \,a(t,\nabla_\xi X_t)\,dt- \tfrac12\; b(t,\xi)\cdot \nabla_\xi(b(t,\xi)\cdot\nabla_\xi X_t)\,dt\hspace{1cm}\\
=b(t,\xi)\cdot\nabla_\xi X_t\,d\beta(t),\quad \mbox{ in } (0,T)\times\mathcal O\end{aligned}\\
X_0=x\quad \mbox{ in } \mathcal O\\
X_t=0 \quad \mbox{ on } (0,T)\times\partial\mathcal O
\end{array}\rk.
\end{equation}
Here $a\colon (0,T)\times\R^d\to \R^d$ is a map of gradient type, i.e.,
\[
a(t,y)=\partial j(t,y),\qquad \forall y\in\R^d, t\in (0,T)
\]
where $j\colon (0,T)\times\R^d\times\Omega\to \R$ is convex in $y$, progressively measurable in
$(t,\omega)\in[0,T)\times\Omega$ and
\begin{equation}
\gamma_1+\alpha_1\,|y|^{p_1}\le j(t,y)\le \gamma_2+\alpha_2\,|y|^{p_2}, \qquad \forall y\in\R^d, \omega\in\Omega, t\in (0,T)
\end{equation}
\begin{equation}\label{e4.3}
j(t,-y)\le c_1\;j(t,y) +c_2,\qquad \forall y\in \R^d, t\in (0,T).
\end{equation}
It should be emphasized that the mapping $r\to a(t,r)$ might be multivalued and discontinuous.
As a matter of fact if $a(t,\cdot)$ is discontinuous at $r=r_j$, but left and right continuous
(as happens by monotonicity) it is replaced by a multivalued maximal monotone mapping
$\tilde a$ obtained by filling the jumps at $r=r_j$.\\
Equation \eqref{e4.1} is of the form \eqref{e1.1} where $H=L^2(\mathcal O)$,  $V=W_0^{1,p_1}(\mathcal O)$,
$A(t)=\partial\psi(t,\cdot)$, $2\le p_1\le p_2<\infty$,
\[
\psi(t,u)=\int_{\mathcal O}j(t,\nabla u)\,d\xi,\qquad \forall u\in W_0^{1,p_1}(\mathcal O)
\]
and
\begin{equation}
B(t)u=b(t,\xi)\cdot\nabla_\xi u=\mbox{div}_\xi (b(t,\xi)u),\qquad \forall u\in W_0^{1,p_1}(\mathcal O).
\end{equation}

As regards the function $b(t,r)\colon [0,T]\times \R^d\to \R^d$ we assume that
\begin{eqnarray}
&&b(t,\cdot),\quad\frac{\partial b}{\partial r}(t,\cdot)\in \big(C([0,T];\bar{\mathcal O})\big)^d\label{e4.5a}\\\nonumber\\
&&r \to b(t,\cdot)+\alpha r \mbox{ is monotone for some }\alpha\ge 0, \label{e4.6a}\\\nonumber\\
&&\mbox{div}_\xi  b(t,\xi))=0,\quad b(t,\xi)\cdot\nu(\xi)=0\quad \forall \xi\in \partial\mathcal O\label{e4.7a}
\end{eqnarray}
where $\nu$ is the normal to $\partial\mathcal O$. (The boundary $\partial\mathcal O$ is
assumed to be of class $C^1$.)\\
Here $\mbox{div}_\xi  b$ is taken in the sense of distributions on $\mathcal O$.\\
Then (4.4) defines a linear continuous operator $B(t)$ from $V$ to
$H=L^2(\mathcal O)$ which as early seen is densely defined
skew-symmetric, that is $-B(t)\subset B^\ast(t)$ $\forall t\in
[0,T]$. Moreover, $B(t)$ is {\it m}-dissipative in $L^2(\mathcal
O)$, that is the range of $u\to u-B(t)u$ is all of $L^2(\mathcal
O)$. Indeed for each $f\in L^2(\mathcal O)$ the equation
$u-B(t)u=f$ has the solution
\[ u(\xi)=\int_0^\infty e^{-s}f(Z(s,\xi))\,ds,\quad \forall \xi\in \mathcal
O,\]
where $s\to Z(s,\xi)$ is the differential flow defined by equation
\begin{equation}\label{e4.8a}
\frac{dZ}{ds}=b(t,Z),\quad s\ge 0,Z(0)=\xi.\end{equation}
(By assumptions (4.6), (4.7), it follows that $t\to Z(t,\xi)$ is well defined on $[0,\infty)$.)\\
Hence, for each $t\in[0,T]$, $B(t)$ generates a $C_0$-group
$(e^{sB(t)})_{s\in\RR}$ on $L^2(\mathcal O)$ which is given by
\[\big(e^{B(t) s}f\big)(\xi)=f(Z(s,\xi)),\quad \forall f\in L^2(\mathcal O), \, s\in\RR.\]
It is also clear that $e^{B(t)s}V\subset V$ for all $s\ge 0$.
\begin{Remark}\em
Assumptions (4.5)--(4.7) can be weakened to discontinuous
multivalued mappings $\xi\to b(t,\xi)$ satisfying (4.6), (4.7) and
such that the solution $Z=Z(s,\xi;t )$ to the characteristic
system (4.8) is differentiable in $t$. The details are omitted.
\end{Remark}
The corresponding random differential equation \eqref{e1.2} has the form
\begin{equation}
\lk\{ \begin{array}{l}
\begin{aligned}\ds\frac{\partial y}{\partial t}- e^{\beta(t)B(t)} \mbox{div}_\xi(a(t,\nabla_\xi e^{-\beta(t)B(t)}y)\hspace{4cm}\\
+\int_0^{\beta(t)}e^{s\,B(t)}\dot B(t)e^{-s\,B(t)}y\,ds=0, \quad\mbox{in } (0,T)\times\mathcal O,\end{aligned}\\
y(0,\xi)=x(\xi)\quad \mbox{in } \mathcal O,\\
y(t,\xi)=0\quad \mbox{on } (0,T)\times \partial\mathcal O.
\end{array}\rk.
\end{equation}
Then by theorem 2.1 we have
\begin{Theorem}
There exists a solution $X$ to \eqref{e4.1} such that $\P$-a.s. $X\in L^{p_1}(0,T;W_0^{1,p_1}(\mathcal O)
)\cap L^{p_2'}(0,T;W^{-1,p_2'}(\mathcal O)) \cap L^{\infty}(0,T;L^{2}(\mathcal O))$.
\end{Theorem}
We also note that in line with Theorem \ref{t2.2} if $X_n$,
$n\in\mathbb{N}$,  are solutions to equations
\begin{equation}\label{eq-44.8}
\lk\{\begin{array}{l}
\begin{aligned}dX ^n_t -\divv_\xi a ^n(t,\nabla_\xi X ^n_t)\, dt-
\tfrac 12 b_n(t,\xi)\cdot \nabla \lk( b _n(t,\xi) X ^n_t\rk) \, dt  = \hspace{0cm}\\
=b _n(t,\xi )\cdot \nabla_\xi X ^n_t \, d\beta ^n(t),\phantom{\Big|}  \quad
(t,\xi) \mbox{ in $(0,T)\times \CO$},\phantom{\Big|}\end{aligned}
\\
X ^n_0= x \mbox{ in } \CO,\phantom{\Big|}
\\
X ^n_t = \mbox{ on } (0,T)\times \partial \CO,\phantom{\Big|}
\end{array}\rk.
\end{equation}
where $b_n\to b$ uniformly on $[0,T]\times \CO$ and $a_n(t,y) \to a(t,y)$,
$a_n ^{-1}(t,y) \to a ^{-1}(t,y)$, $\beta_n(t)\to \beta(t)$ for all $y\in\RR ^d$, $t\in [0,T)$,
then $X ^n\to X$ weakly in $L^{p_1}(0,T;W ^{1,p_1}_0 (\CO))$.
Standard examples refer to structural stability, PDEs as well as
to homogenization type results for equation \eqref{e4.1}. In latter case $a_n(t,z)=a(t,n z)$
where $a(t,\cdot)$ is periodic (see e.g., \cite{ABM}.)

Equation \eqref{e4.1} is relevant in the mathematical
description of nonlinear diffusion processes  perturbed by a Brownian
distribution with coefficient transport term
$b(t,\xi)\cdot\nabla_\xi X$.

The assumption $p_1\ge 2$ was taken here for technical
reason required by the functional framework we work in and this excludes several relevant examples.
For instance, the limit case $p_1=1$ which corresponds to the
nonlinear diffusion function $a(t,y)= \rho \frac{y}{|y|_d}$,
$\rho>0$, which is relevant in material science and image
restoring techniques (see e.g. \cite{barbu-barbu,barbu-2010}) is
beyond our approach and requires a specific treatment
(see also \cite{barbu-dp-R}
for the treatment of a similar problem with additive and continuous noise.)

In 2-D the appropriate functional setting to treat such a problem
is $V=BV(\mathcal O)$ the space of functions with bounded
variation on $\mathcal O$ with the norm $\varphi(y)$ and
$H=L^2(\mathcal O)$. Here
$\varphi(y)=\|Dy\|+\int_{\partial \mathcal O}|\gamma_0(y)| d\mathcal H$, $y\in V$,
$\|Dy\|$ is  the variation of $y\in V$, $\gamma_0(y)$ is the trace on $\partial\mathcal O$
and $d\mathcal H$ is the Hausdorff measure on $\partial\mathcal O$.
We recall that the norm $\varphi$ is just the lower semicontinuous closure of the norm
of Sobolev space $W_0^{1,1}\mathcal O)$  (see e.g., \cite[pag. 438]{A}.)
Then the approach developed in section 3
can be adapted to present situation though $V$ is not reflexive.
We expect to treat this limit case in a forthcoming work. (On
these lines see also \cite{B2}.)

\subsection{Linear diffusion equations with nonlinear Neumann boundary conditions}
Consider the equation
\begin{equation}\label{e4.8}\lk\{
\begin{array}{l}\begin{aligned}
dX_t-\Delta X_t\,dt-\frac12 b(t,\xi)\cdot\nabla_\xi\big(b(t,\xi)\cdot\nabla_\xi X_t\big)\,dt
=\\b(t,\xi)\cdot\nabla_\xi X_t\;d\beta(t) \quad \mbox{in }[0,T]\times\mathcal O\end{aligned}\\
\ds \frac{\partial\ }{\partial\nu}X_t+\zeta(t,X_t)\ni 0\quad \mbox{on }[0,T]\times\partial\mathcal O\\
X_0=x\qquad \mbox{in } \mathcal O
\end{array}\rk.
\end{equation}
where $\zeta(t,r)=\partial j_0(t,r)$, $\forall t\in(0,T)$, $r\in\R$ and $j_0(t,\cdot)$
is a lower semicontinuous convex function on $\R$ such that
\[
\gamma_1+\alpha_1|y|^2\le j_0(t,y)\le\gamma_2+\alpha_2|y|^2,\qquad \forall y\in\R, t\in(0,T)
\]
and $\alpha_i>0$, $\gamma_i\in\R$, $i=1,2$.\\
Assume also that \eqref{e4.3} holds and that $b=b(t,\cdot)$ satisfies  conditions \eqref{e4.5a}--\eqref{e4.7a}.\\
Then we may apply Theorems \ref{t2.1}, \ref{t2.2}, and \ref{t2.3},
where $V=H^1(\mathcal O)$, $H=L^2(\mathcal O)$ and
\[
\psi(t,y)=\tfrac12 \int_{\mathcal O}|\nabla y|^2\,d\xi+ \int_{\partial\mathcal O}j(t,y)\,d\xi,\qquad \forall y\in V.
\]
It follows so the existence of a solution $X\in L^2(0,T;V)\cap
W^{1,2}([0,T];V')$ to \eqref{e4.8} and also the structural
stability of \eqref{e4.8} with respect to $b$. Problems of this
type arise in thermostat central. In this case
\[
\zeta(t,y)= \begin{cases}\big(\alpha_1(t) H(y)+\alpha_2(t)H(-y)\big)y&\mbox{if }y\not=0\\
[-\alpha_2(t),\alpha_1(t)]&\mbox{if } y=0
\end{cases}
\]
where $\alpha_i>0$, $\forall t\in[0,T[$ and $H$ is the Heaviside function.


\subsection{Nonlinear stochastic porous media equation}
Consider the equation
\begin{equation}\label{e4.6}\lk\{
\begin{array}{l}
\begin{aligned}dX_t-\Delta_\xi \phi(t,X_t)\,dt- \tfrac12\; b(t,\xi)\cdot \nabla_\xi (-\Delta)^{-1}(b(t,\xi)\cdot\nabla_\xi ((-\Delta)^{-1}X_t)\,dt\hspace{-.6cm}\\
=b(t,\xi)\cdot\nabla_\xi (-\Delta)^{-1}X_t\,d\beta(t),\quad \mbox{ in } (0,T)\times\mathcal O\end{aligned}\\
X_0=x\quad \mbox{ in } \mathcal O\\
X_t=0 \quad \mbox{ on } (0,\infty)\times\partial\mathcal O
\end{array}\rk.
\end{equation}
Here $\mathcal O\subset \R^d$, $d=1,2,3$ is a bounded open domain
and $(-\Delta)^{-1}$ is the inverse of the operator $A_0=-\Delta$,
$D(A_0)=H_0^1(\mathcal O)\cap H^2(\mathcal O)$.
The function $\phi\colon (0,T)\times\R^d\to\R$ is assumed to satisfy the following conditions\\[6pt]
\textbf{(k)} \emph{$\phi=\phi(t,r)$ is monotonically decreasing in $r$,  measurable in $t$ and its
potential
\[
j(t,r)=\int_0^r \phi(t,\tau)\,d\tau, \qquad t\in(0,T)
\]
satisfies the growth conditions
\begin{equation}
\gamma_1+\alpha_1\,|r|^{p_1}\le j(t,r)\le \gamma_2+\alpha_2\,|r|^{p_2}, \qquad \forall r\in\R, \omega\in\Omega, t\in [0,T]
\end{equation}
\begin{equation}
j(t,-r)\le c_1\;j(t,r) +c_2,\qquad \forall r\in \R , t\in (0,T)
\end{equation}
where $\frac65\le p_1\le p_2<\infty$ if $d=3$, $1<p_1\le p_2<\infty$ if $d=1,2$.} \\[6pt]
Then equation \eqref{e4.6} can be written as \eqref{e1.1}, where
$H=H^{-1}(\mathcal O)$, $V=L^{p_1}(\mathcal O)$ and
$A(t)=\partial\psi(t,\cdot)$ where $\psi(t,\cdot)\colon
H\to\bar\R$ is defined by
\[
\psi(t,y)=
\begin{cases}
\ds\int_{\mathcal O} j(t,y)\,d\xi&\mbox{if } y\in H^{-1}(\mathcal O), j(t,y)\in L^1(\mathcal O)\\
+\infty &\mbox{otherwise},
\end{cases}
\]
and $B(t)$, $t\in\R^+$ is defined by
\begin{equation}\label{e4.15a}
B(t)u=b(t,\xi)\cdot \nabla((-\Delta)^{-1}u),\quad u\in V.
\end{equation}
The space $V'$ is in this case the dual of $V=L^{p_1}(\mathcal O)$ with  $H^{-1}(\mathcal O)$
as \emph{pivot} space. By the Sobolev embedding theorem it is easily seen that since $p_1\ge\tfrac65$
we have $V\subset H^{-1}(\mathcal O)$. The scalar product on $H$ is defined by
\[
\langle u,v\rangle_{H^{-1}(\mathcal O)}=u(z), \quad z=(-\Delta)^{-1}v.
\]
It is well known that $A(t) X=-\Delta_\xi\phi(t,X)$ is indeed the subdifferential of $\psi(t,\cdot)$
in $H^{-1}(\mathcal O)$ (see e.g., \cite[pag. 68]{barbu-2010}).\\
As regards $b\colon [0,T]\times \bar{\mathcal O}\to\R^d$
we assume that conditions \eqref{e4.5a}--\eqref{e4.7a} hold.\\
We note that for each $t\in [0,T]$, $B(t)\in L(V,H^{-1}(\mathcal
O))$ is densely defined and skew-symmetric on $H^{-1}(\mathcal
O)=H$. Indeed we have
\[
\begin{split}\langle B(t)u,u\rangle=\int_{\mathcal O}  \mbox{div}(b(t,\xi)(-\Delta)^{-1}u)\cdot(-\Delta)^{-1}u \,d\xi=\\
=\tfrac12\int_{\mathcal O}
b(t,\xi)\cdot\nabla|(-\Delta)^{-1}u(\xi)|^2 \,d\xi=0,\end{split}\]
because $\mbox{div}_\xi b=0$ and $b(t,\xi)\cdot\nu(\xi)=0$ on $\partial\mathcal O$.\\
Moreover, for each $t\in[0,T]$, $B(t)$ is {\it m}-dissipative on $H^{-1}(\mathcal O)$. Indeed for each
$f\in H^{-1}(\mathcal O)$, the equation $u-B(t)u=f$ can be equivalently written as $v=(-\Delta)^{-1}u$, where
\[
\begin{array}{rcl}
-\Delta v-b(t,\cdot)\cdot\nabla v&=&f\quad \mbox{in }\mathcal O,\\
v&=&0\quad \mbox{on }\partial\mathcal O.
\end{array}
\]
By Lax-Milgram lemma the latter has a unique solution $v\in H_0^1(\mathcal O)$ and therefore
$u\in H^{-1}(\mathcal O)$ as claimed. Moreover, if $f\in L^{p_1}(\mathcal O)$
and $\partial\mathcal O$ is of class $C^2$
then by the Agmon-Douglis-Nirenberg theorem $v\in W^{2,p_1}(\mathcal O)\cap W_0^{1,p_1}(\mathcal O)$ and so $u\in V$.\\
Hence, $B(t)$ generates a $C_0$-group $(e^{sB(t)})_{s\in\RR}$ on
$H=H^{-1}(\mathcal O)$ which leaves $V=L^{p_1}(\mathcal O)$
invariant.
\\
Then we may apply Theorem \ref{t2.1} as well as the approximation
Theorem \ref{t2.2} to the present situation. We obtain
\begin{Theorem}
There is a unique solution $X$ to \eqref{e4.6} such that $\P$-a.s.
$X\in L^{p_1}(0,T;L^{p_1}(\mathcal O))\cap L^\infty(0,T;H^{-1}(\mathcal O))$.
Moreover, the solution $X$ is a limit of approximating solutions when the Brownian motion $\beta$ is approximated by a sequence of smooth processes.
\end{Theorem}
Moreover, if $\phi_n\to \phi$ and $\phi_n^*\to \phi^*$, $b_n\to b$ we find by Theorem 2.2 that the
corresponding solutions $X_n$ to \eqref{e4.6} are convergent to solution $X$ to \eqref{e4.8}.
The details are omitted.

Existence for stochastic porous media equation of the form
\[
\lk\{
\begin{array}{l}
dX_t-\Delta_\xi \phi(X_t)\,dt=\sigma(X_t)\,dW(t) \quad \mbox{in } (0,T)\times\mathcal O\\
X_0=x \quad \mbox{ in } \mathcal O\\
X_t=0 \quad \mbox{ on } (0,T)\times\partial\mathcal O,
\end{array}
\rk.
\]
when $W_t$ is a Wiener process
of the form
\[
W(t,\xi)=\sum_{k=1}^\infty \mu_k\,e_k(\xi) \,\beta_k(t)
\]
with $\sum_{k=1}^\infty \mu_k^2\,\lambda_k^2<\infty$, $\Delta e_k=-\lambda_k\,e_k$ in $\mathcal O$, $e_k\in H_0^1(\mathcal O)$,
and $\sigma=\sigma(x)$ is a linear continuous operator, were studied
in \cite{barbu-R,barbu-dp-R}. Note that in this case the noise term can also be written in our form
with commuting and, contrary to our paper, bounded operators $B_j$. Here the multiplicative term
$\sigma(X_t)=b\cdot\nabla (-\Delta)^{-1}X_t$ is however
discontinuous on the space $H^{-1}(\mathcal O)$ and so Theorem 4.2
is from this point of view different and in this sense more general.

Equation \eqref{e4.6} models diffusion processes and the motion of fluid flows in porous media.
The case considered here $(p_1> 1)$ is that of slow diffusion.

\begin{Remark}\label{rem-Sussman-pm}\em
Theorem \ref{thm-sussmann} and Remark \ref{rem-sussmann} are also valid in the current setup.
\end{Remark}

\appendix
\section{Convex functions}
\label{Bapp}\label{appendix-convex-function} We summarize in this
paragraph some facts about convex functions, which we have used in
our paper.

Given a convex and lower-semi continuous  function $\phi:Y\to\bar
\RR =(-\infty,\infty]$ we denote by $\partial \phi:Y\to Y'$ (the
dual space) the {\sl subdifferential} of $\phi$, i.e.\ \DEQSZ
\label{eq-1.6}
\partial \phi (y) := \lk\{z\in Y': \phi(y)-\phi(u)\le \lb y-u,z\rb ,\, \, \forall u\in Y\rk\}.
\EEQSZ (Here $\lb \cdot,\cdot\rb $ is the duality paring between
$Y$ and $Y'$).
 The function $\phi ^ \ast:Y'\to Y$ defined by
\DEQSZ\label{eq-1.7} \phi ^ \ast (z)&=& \sup\lk\{\lb y,z\rb -
\phi(y): y\in Y\rk\}, \EEQSZ is called the conjugate of $\phi$ and
as
 $\phi$. Similarly to
it is convex lower semi continuous function on $Y'$. Also we
notice the following key conjugacy formulae (see e.g.\ \cite[p.
89]{barbu-1987}). If   $y\in Y$, and $z\in Y'$
\DEQSZ
\label{eq-1.9} \phi(y)+\phi ^ \ast (z) &\ge & \lb y,z\rb \quad
\EEQSZ
 \DEQSZ\label{eq-1.8} \phi(y)+\phi ^ \ast (z) &=& \lb y,z\rb
\quad \mbox{ iff } z\in \partial \phi(y), \EEQSZ

A vector $x^\ast$ is said to be a subgradient of a convex function
$\phi$ at a point $x$ if \DEQSZ\label{subg} \phi(z) \ge \phi(x) +
\lb x^\ast,z-x\rb. \EEQSZ

Moreover, straightforward calculations give
\DEQSZ\label{shift_con} \phi^\ast(x^\ast) = \phi_y^\ast(x^\ast) -
(y,x^\ast) , \EEQSZ whenever $\phi(x)=\phi_y(x+y)$.

\medskip
\begin{acknowledgements}
The work of V. Barbu was supported by the grant of the Romanian National
Authority for Scientific Research  1ERC/02.07.2012
and by CIRM (Fondazione Bruno Kessler). The
work of E. Hausenblas was supported  by the Austrian Science Fund
(FWF): P20705. Moreover, the authors would like to thank the
Newton Institute, where part of this work was done during the
special semester on ``Stochastic Partial Differential Equations".
\end{acknowledgements}


\end {document}